\documentclass[11.5pt,oneside,leqno]{amsart}
\numberwithin{equation}{section}
\usepackage[utf8]{inputenc}
\usepackage{url}
\usepackage{mathrsfs}
\usepackage{bbm}
\usepackage{amssymb}
\usepackage{hyperref}
\usepackage{amsmath}
\usepackage{mathabx}
\usepackage{bigints}
\usepackage{bbold}
\usepackage{xcolor}
\usepackage{graphicx}

\def\XXint#1#2#3{{\setbox0=\hbox{$#1{#2#3}{\int}$ }
\vcenter{\hbox{$#2#3$ }}\kern-.6\wd0}}

\newtheorem{theorem}{Theorem}[section]
\newtheorem{corollary}{Corollary}
\newtheorem{lemma}[theorem]{Lemma}
\newtheorem{conjecture}[theorem]{Conjecture}
\theoremstyle{definition}

\newtheorem{remark}[theorem]{Remark}
\theoremstyle{proposition}

\title{Some Mizohata-Takeuchi-type estimate for exponential sums}
\author{Xuerui Yang}

\address{ Xuerui Yang\\
Department of Mathematics\\
University of Illinois at Urbana-Champaign\\
Urbana, IL, 61801, USA}

\email{xueruiy3@illinois.edu}

\begin{document}

\begin{abstract}
Let $R^{\frac{1}{2}}$ be a large integer, and $\omega$ be a nonnegative weight in the $R$-ball $B_R=[0,R]^2$ such that $\omega(B_R)\le R$. For any complex sequence $\{a_n\}$, define the quadratic exponential sum
\[
G(x,t)=\sum_{n=1}^{R^{\frac{1}{2}}} a_n e\big(\frac{n}{R^{\frac{1}{2}}} x+\frac{n^2}{R} t\big).
\]
It holds that   
\[
\int |G|^2 \omega \lessapprox \sup_{T}\omega(T)^{\frac{1}{2}}\cdot R \,\|a_n\|_{l^2}^2
\]
where $T$ ranges over $R\times R^{\frac{1}{2}}$ tubes in $B_R$. The proof is established through exploring the distributions of superlevel sets of the $G$ function. It is based on the $TT^*$ method and the circle method. 
\end{abstract}

\maketitle

\section{Introduction}

Given a convex $C^2$ function $\Gamma$ on $[0,1]$, define $E_\Gamma$ to be the extension operator for the curve $\big\{(\xi,\Gamma(\xi)):\xi \in [0,1]\big\}$. That is, given $g\in L^2([0,1])$, 
\[
E_\Gamma g(x,t)=\int_{[0,1]} e\big(x\xi +t\Gamma(\xi)\big) g(\xi) d\xi. 
\]
The Mizohata-Takeuchi Conjecture \cite[Conj. 1.5]{HC} states that 
\begin{conjecture} For any nonnegative weight $\omega$ in $B_R$, we have 
\[
\int_{B_R} |E_\Gamma g|^2 \omega \lessapprox \sup_{P} \omega(P)\cdot \|g\|_2^2  ,
\]
where $\omega(P)=\int_P \omega$, and $P$ is any infinite tube of width $1$ in the plane. 
\end{conjecture}

Using refined decoupling theorem \cite[Thm 4.2]{GIOW}, A. Carbery, M. Iliopoulou and H. Wang \cite{CIW} showed that 
\begin{theorem}[CIW] For all weights $\omega: \mathbb R^2\to [0,\infty)$ we have 
\[
\int_{B_R} |E_\Gamma g|^2 \omega \lessapprox \sup_{T} \big(\int_{T} \omega^{\frac{3}{2}} \big)^{\frac{2}{3}}  \|g\|_{L^2}^2, 
\]   
where $T$ ranges over all $R\times R^{\frac{1}{2}}$ tubes. 
\end{theorem}
A. Ortiz considered a similar problem \cite{O} with the extension operator for the truncated cone in $\mathbb R^3$, and showed stronger Mizohata-Takeuchi-type estimates for one-dimensional weights in $\mathbb R^3$. 

For a one-dimensional weight in the plane, we expect a better estimate than Carbery-Iliopoulou-Wang's result.  
\begin{conjecture}  \label{MT general}
Let $\omega: B_R\to [0,\infty)$ be a one-dimensional weight, that is, 
\[
\omega\big(B(y,r) \big)\le r
\]
for any ball of radius $r$ centered at $y$, $1\le r\le R$. Then  
\begin{equation} 
\int_{B_R} |E_\Gamma g|^2 \omega \lesssim_\epsilon R^\epsilon \sup_{T}\omega(T)^{\frac{1}{2}} \|g\|_2^2,    
\end{equation}
where $T$ ranges over $R\times R^{\frac{1}{2}}$ tubes. 
\end{conjecture}
However, as we will see from Lemma \ref{counterex}, this conjecture is false for general $C^2$ convex curves. Therefore, we turn to the special case when $\Gamma(\xi)=\xi^2$ is the quadratic function, and the curve is the truncated parabola. Define $E$ to be the extension operator for the truncated parabola, it is conjectured that 
\begin{conjecture}  \label{MT 1d}
Let $\omega: B_R \to [0,\infty)$ be a one-dimensional weight, then it holds that 
\begin{equation}  \label{key inequa}
\int_{B_R} |E g|^2 \omega \lesssim_\epsilon R^\epsilon \sup_{T}\omega(T)^{\frac{1}{2}} \|g\|_2^2,    
\end{equation}
where $T$ ranges over $R\times R^{\frac{1}{2}}$ tubes.     
\end{conjecture}

This conjecture is sharp up to an $R^\epsilon$ factor due to the Knapp example and the quadratic Weyl sum example. Consider 
\[
Eg=\sum_{n=1}^{R^{\frac{1}{2}}} e\big( \frac{n}{R^{\frac{1}{2}}}x+\frac{n^2}{R} t \big)
\]
in $B_R$, and define $U$ to be the union of unit boxes centered at 
\[
\big\{(x,t)\in B_R: \frac{x}{R^{\frac{1}{2}}}\equiv \frac{k}{q} (\text{mod }1), \frac{t}{R}=\frac{a}{q}, q\sim R^{\frac{1}{6}}, 0\le a, k<q, (a,q)=1 \big\}.
\]
We see that $U$ is a one-dimensional set in $B_R$, and $|Eg|\sim R^{\frac{5}{12}}$ on an $R^{-\epsilon}$ proportion of $U$ for any small $\epsilon>0$. Thus for \eqref{key inequa} to be true with $\omega=1_U$, the tube $T$ we choose has to be horizontal. However, the horizontal tubes do not show up in the wave packet decomposition of $Eg$. As a matter of fact, we will see that in the case when $Eg$ is an exponential sum, it suffices to consider horizontal tubes in \eqref{key inequa}. 

Our result is to prove the special case of Conjecture \ref{MT 1d} when $G=Eg$ is an quadratic exponential sum, namely, 
\begin{equation}  \label{def G}
G(x,t)=\sum_{n=1}^{R^{\frac{1}{2}}} a_n e\big(\frac{n}{R^{\frac{1}{2}}} x+\frac{n^2}{R} t\big).    
\end{equation}
in $B_R$. In this case, 
\[
\|g\|_2^2\sim R^{-1} \|G\|_{L^2(B_R)}^2\sim R \, \|a_n\|_{l^2}^2. 
\]
Then we can obtain the following estimate. 
\begin{theorem}  \label{main thm}
Suppose that $R^{\frac{1}{2}}$ is an integer, and let $\omega$ be a nonnegative weight in $B_R$ such that $\omega(B_R)\le R$. For any complex sequence $\{a_n\}_{n=1}^{R^{\frac{1}{2}}}$, define the exponential sum $G$ as above \eqref{def G}. Then it holds that 
\[
\int_{B_R} |G|^2 \omega \lessapprox \sup_{T: R\times R^{\frac{1}{2}} \text{tube}} \omega(T)^{\frac{1}{2}} \cdot R \, \|a_n\|_{l^2}^2.  
\]
\end{theorem}
The proof of Theorem \ref{main thm} relies on the distributions of superlevel sets of the function $G$. 
\begin{theorem}    \label{L4 thm}
Let $N$ be a large integer and $\{a_n\}$ be a complex sequence. Consider the quadratic sum 
\[
f(x,t)=\sum_{n=1}^N a_n e(nx+n^2 t) 
\]
defined on the torus $\mathbb T^2$. In each horizontal strip 
\[
S_j=[0,1]\times \big[\frac{j-1}{N},\frac{j}{N}\big] \subset [0,1]^2, \quad j\in [1,N]\cap \mathbb Z,
\]
we pick a $\frac{1}{N}\times \frac{1}{N^2}$ box $B_j$ (this is the largest box preserving the locally constant property of a quadratic sum), and define $E=\bigcup_{j=1}^N B_j$ to be the union of these boxes. Then we have
\begin{equation}  \label{l4}
\Big\|\sum_{n=1}^N a_n e(nx+n^2 t)\Big\|_{L^4(E)} \lessapprox N^{-\frac{1}{4}} \|a_n\|_{l^2}.
\end{equation}   
\end{theorem}
This $L^4$ estimate \eqref{l4} is sharp by considering $a_n=1$ for all $n$ and $|f|\sim N$ in the $\frac{1}{N}\times \frac{1}{N^2}$ box centered at origin. Moreover, it has the following corollary regarding the distribution of the superlevel sets of $f$.  
\begin{corollary}  \label{dist of lvl sets}
Following the definitions in Theorem \ref{L4 thm}, let $\lambda \in [N^{\frac{1}{4}},N^{\frac{1}{2}}]$ be a dyadic number. Define $\#_{\lambda}$ to be the number of horizontal strips $S_j$ such that there is a point $(x_j,t_j)$ in $S_j$ with $|f(x_j,t_j)|\sim \lambda  \|a_n\|_{l^2}$, then 
\[
\#_\lambda\lessapprox N^2 \lambda^{-4}. 
\]
\end{corollary}
This level set estimate is sharp up to an $N^\epsilon$ factor for every $\lambda$ by considering the quadratic Weyl sum. It follows directly from Theorem \ref{L4 thm} and the fact that $|f|$ is locally constant on each box $B_j$. We omit the proof. 

\,

\textbf{Notations}: $e(z)=e^{2\pi i z}$ is the complex exponential, here $z\in \mathbb R$. 

$A\lesssim B$ (or $A=O(B)$) means $A \le C B$,  where $C$ is some positive constant, and $A\lesssim_\epsilon B$ (or $A=O_\epsilon(B)$) indicates that the implicit constant may depend on the subscript $\epsilon$. 

$A\lessapprox B$ stands for $A\lesssim_\epsilon R^\epsilon B$. 

$A\sim B$ means that we have both $A\lesssim B$ and $B\lesssim A$. 

To distinguish the Fourier transforms on $\mathbb T^2$ and $\mathbb R^2$, we use $\mathcal{F}$ to denote the Fourier transform on $\mathbb T^2$, and $\hat{}$ to denote the Fourier transform on $\mathbb R^2$. 

\,

{\bf Acknowledgements.} The author would like to thank Shukun Wu and Alex Ortiz for introducing problem \ref{MT general} to him. Shukun Wu also pointed out the idea that leads to Lemma \ref{counterex}. The author is also grateful to Xiaochun Li and Zane Li for some helpful discussions. The author is supported by a UIUC department fellowship.

\section{The locally constant property}

The locally constant property is a fact that we will use throughout this paper. In this section, we show its validity. 

For the exponential sum 
\[
f(x,t)=\sum_{n=1}^N a_n e(nx+n^2 t) 
\]
defined on the torus $\mathbb T^2$, the Fourier transform of $f$, denoted by $\mathcal{F}f$, is supported on $\mathbb Z^2\cap \Big( [0,N]\times [0,N^2]\Big)$. 

There is a smooth bump function $\eta: \mathbb R^2 \to \mathbb R^2$ such that $|\eta|\le 1$, $\eta=1$ on $[-2N,2N]\times [-2N^2,2N^2]$, and $\hat{\eta}$ is compactly supported on $[-10N^{-1},10N^{-1}]\times [-10 N^{-2},10N^{-2}]$. $\eta$ can be viewed as a function on $\mathbb Z^2$ and $\hat{\eta}$ can be viewed as a function on $\mathbb T^2$. 

Therefore, 
\[
\mathcal{F}f(n_1,n_2)=\mathcal{F}f(n_1,n_2)\cdot \eta(n_1,n_2). 
\]
We apply inverse Fourier transform to see that 
\[
\begin{split}
f(x,t)=&f\ast_{\mathbb T^2} \hat{\eta}(x,t)
\\ =& \int_{\mathbb T^2} f(x-y,t-s)\hat{\eta}(y,s) dyds. 
\end{split}
\]
Because of the support of $\hat{\eta}$, we have  
\begin{equation}  \label{lcp}
|f(x,t)|\lesssim \Big( N^3 \int_{B(x,t)} |f(x,t)|^p dxdt \Big)^{\frac{1}{p}}    
\end{equation}
for any $p\ge 1$, where $B(x,t)$ is a $10 N^{-1}\times 10 N^{-2}$ box centered at $(x,t)$ in $\mathbb T^2$. Sometimes, we may omit the constant $10$. It shows that we can bound a local supremum of $|f|$ by the average value of the $L^p$ norm of $|f|$ in the $N^{-1}\times N^{-2}$-neighborhood of point attaining the supremum. This is the locally constant property we are referring to.

\section{Proof of Theorem \ref{main thm}} 

In this section, we use Corollary \ref{dist of lvl sets} to prove our main Theorem \ref{main thm}. After rescaling, and use definition \eqref{def G}, we can restate Corollary \ref{dist of lvl sets} as the following. 
\begin{corollary}  \label{coro reform}
Suppose that $R^{\frac{1}{2}}=N$ is an integer, for $j\in [1,N]\cap \mathbb Z$, define 
\[
\overline{S_j}=[0,R^{\frac{1}{2}}]\times \big[(j-1)R^{\frac{1}{2}},j R^{\frac{1}{2}} \big]
\]
to be a box of radius $R^{\frac{1}{2}}$ in the vertical rectangle $[0,R^{\frac{1}{2}}]\times [0,R]$. Let $\mu\in [R^{\frac{1}{8}},R^{\frac{1}{4}}]$ be a dyadic number, then the number of boxes $\overline{S_j}$ such that there is a point $(x_j,t_j)$ in $\overline{S_j}$ satisfying $|G(x_j,t_j)|\sim \mu \|a_n\|_{l^2}$ is bounded as 
\[
\lessapprox R \mu^{-4}. 
\]
\end{corollary}
\begin{proof}[Proof of Theorem \ref{main thm}]
We first decompose the weight 
\[
\omega=\sum_{\substack{\mu\le R^{\frac{1}{4}}\\ \mu \, \text{dyadic}}}\omega_\mu, 
\]
so that $|G|\sim \mu \|a_n\|_{l^2}$ on the support of $\omega_\mu$. Now it suffices to consider a particular $\mu$ since those $\mu < R^{-10}$ make little contribution in \eqref{key inequa}, and the number of $\mu\in [R^{-10},R^{\frac{1}{4}}]$ is $\lessapprox 1$. By our assumption on $\omega$, $\omega_\mu(B_R)\le R$. Our goal is to show that 
\begin{equation}
\mu^2 \omega_\mu(B_R)\lessapprox R \sup_{T: R\times R^{\frac{1}{2}} \text{tube}} \omega_\mu(T)^{\frac{1}{2}}.
\end{equation}

From now on, $T$ is always an $R\times R^{\frac{1}{2}}$ tube. 
 
If $\mu < R^{\frac{1}{8}}$, then we use
\[
\sup_{T} \omega_\mu(T)\ge R^{-\frac{1}{2}} \omega_\mu(B_R), 
\]
since the right hand side is the average value of $\omega_\mu(T)$ when $T$ varies. By the assumption that $\omega_{\mu}(B_R)\le R$, 
\[
\begin{split}
\mu^2 \omega_\mu(B_R)  \le &  R^{\frac{1}{4}}\cdot \omega_\mu(B_R)^{\frac{1}{2}+\frac{1}{2}}  
\\ \le & R^{\frac{3}{4}} \omega_\mu(B_R)^{\frac{1}{2}}
\\ \le & R \sup_{T} \omega_\mu(T)^{\frac{1}{2}}. 
\end{split}
\]

Now it suffices to consider $\mu \in [R^{\frac{1}{8}},R^{\frac{1}{4}}]$. Define the horizontal tubes 
\[
T_j=[0,R]\times [(j-1)R^{\frac{1}{2}}, jR^{\frac{1}{2}}], \quad j\in[1,R^{\frac{1}{2}}]\cap \mathbb Z.  
\]
By Corollary \ref{coro reform} and the periodicity of $G$ in the $x$-direction, we know that there are 
\[
\lessapprox R \mu^{-4}
\]
tubes $T_j$ s.t. $\omega_\mu(T_j)\neq 0$, therefore, 
\[
\omega_\mu(B_R)\lessapprox R\mu^{-4} \sup_{T} \omega_\mu(T). 
\]
Then 
\[
\begin{split}
\mu^2\omega_\mu(B_R)= & \mu^2 \omega_\mu(B_R)^{\frac{1}{2}+\frac{1}{2}}   
\\ \le & \omega_\mu(B_R)^{\frac{1}{2}} \cdot R^{\frac{1}{2}} \sup_{T} \omega_\mu(T)^{\frac{1}{2}} 
\\ \le & R \sup_{T} \omega_\mu(T)^{\frac{1}{2}}, 
\end{split}
\]
again due to $\omega_\mu(B_R)\le R$. We are done with the proof. 
\end{proof}

\section{Lemmas}

We need two lemmas to prove Theorem \ref{L4 thm}. The first one concerns some incidence estimates about rational numbers.
\begin{lemma}  \label{inci lemma}
Fix a dyadic $Q\in [1,N]$ and let $\mathcal{S}_Q$ be the set of all reduced fractions $\frac{a}{q}\in [-1,1]$ such that $q\sim Q$. We cut the unit interval $I=[0,1]=\bigcup_{j=1}^N I_j$ into $N$ subintervals, where  
\[
I_j=\big[\frac{j-1}{N},\frac{j}{N}\big],  \quad j\in [1,N]\cap \mathbb N, 
\]
and pick a point $t_j$ from each $I_j$. Each triple $(i,j,\frac{a}{q})$ such that 
\[
t_i-t_j=\frac{a}{q}+O(\frac{1}{QN}) \text{ with } \frac{a}{q} \in \mathcal{S}_Q
\]
is called an incidence.  Let $1\le M\le N$ and consider a subset $W_M$ of $[1,N]\cap \mathbb Z$. Then the number of incidences $(i,j,\frac{a}{q})$ with $i,j\in W_M$ is bounded by $C_\epsilon N^\epsilon QM$.   
\end{lemma}
\begin{remark}
This lemma is sharp up to a $N^\epsilon$ factor for $M\ge 10 Q$, as we can fix some $q\sim Q$ and consider the set of points 
\[
\begin{split}
& 0,\frac{1}{q},\frac{2}{q},\dots, \frac{q-1}{q},   
\\ & \frac{1}{N},\frac{1}{N}+\frac{1}{q},\dots, \frac{1}{N}+\frac{q-1}{q},
\\ & \dots,
\\ & \frac{m}{N},\frac{m}{N}+\frac{1}{q},\cdots, \frac{m}{N}+\frac{q-1}{q},  
\end{split}
\]
where $m\le \frac{M}{10Q}$. The number of incidences with $t_i,t_j$ in each row is $\approx Q^2$, and there are $\sim \frac{M}{Q}$ rows.     
\end{remark}
\begin{proof}
We can find a smooth function $\psi$ such that $\psi\ge 0$, $\hat{\psi}\ge 0$, $\hat{\psi}\ge 1$ on $[-1,1]$ and $\hat{\psi}$ is compactly supported on $[-2,2]$. Such a function is constructed in the proof of Lemma 7 in Bloom and Maynard's paper \cite{BM}. Then we relax the condition 
\[
t_i-t_j=\frac{a}{q}+O(\frac{1}{QN})
\]
to that 
\[
t_i-t_j-\frac{a}{q} \equiv O(\frac{1}{QN}) (\text{mod }1).
\]
Since the left hand side is absolutely bounded by $2$, this relaxation of condition is harmless. Now existence of an incidence can be detected by the following exponential sum
\[
\frac{1}{QN}\sum_{n}\psi\big(\frac{n}{QN}\big) e\big((x_i-x_j-\frac{a}{q})n \big)=\sum_{k} \hat{\psi}\big( \frac{x_i-x_j-\frac{a}{q}-k}{\frac{1}{QN}} \big),
\]
since we assume that $\hat{\psi}$ is nonnegative and $\ge 1$ on $[-1,1]$. By symmetry, it suffices to consider $\frac{a}{q}\ge 0$. Therefore the number of incidences is bounded above by 
\begin{equation}  \label{counting}
\begin{split}
& \sum_{q\sim Q} \sum_{\substack{0\le a< q\\ (a,q)=1}} \sum_{i,j\in W_M} \frac{1}{QN} \sum_{n}\psi\big(\frac{n}{QN}\big) e\big((t_i-t_j-\frac{a}{q})n \big)
\\ =& \frac{1}{QN} \sum_{n} \psi\big(\frac{n}{QN}\big) \big|\sum_{i\in W_M} e(t_i n) \big|^2 \sum_{\substack{0\le a< q\\ (a,q)=1}}  e(-\frac{a}{q}n). 
\end{split}    
\end{equation}
 
We use the standard Ramanujan's sum notation 
\[
c_q(n)=\sum_{\substack{0\le a< q\\ (a,q)=1}}  e(-\frac{a}{q}n)=\mu\big( \frac{q}{(q,n)} \big) \frac{\phi(q)}{\phi(\frac{q}{(q,n)})},
\]
where $\mu$ is the M\"obius function and $\phi$ is Euler's totient function. 

Therefore, 
\[
\big|c_q(n)\big|\le (q,n), 
\]
and thus 
\[
\sum_{q\sim Q} c_q(n)\lesssim_\epsilon N^{\epsilon }Q
\]
for $n\neq 0$. In the above sum \eqref{counting}, we can distinguish the cases when $n\neq 0$ and $n=0$. Indeed, since we also assume that $\psi$ is nonnegative, the second line of \eqref{counting} is 
\[
\begin{split}
\lessapprox & \frac{1}{QN} M^2 Q^2  \qquad \qquad \qquad (n=0)
\\+& \frac{1}{QN} \sum_{n} \psi\big(\frac{n}{QN}\big) \big|\sum_{i\in W_M} e(t_i n) \big|^2 \cdot Q  
\\ \lessapprox & QM+ Q\cdot \#\{(i,j)\in W_M^2:\|t_i-t_j\|\le \frac{1}{QN}\}
\\ \lessapprox & QM. 
\end{split}
\]
\end{proof}

The second lemma tells us that the $L^4$ estimate \eqref{l4} is equivalent to a weighted $L^2$ estimate. Let us recall the definitions in Theorem \ref{L4 thm}. The quadratic sum 
\[
f(x,t)=\sum_{n=1}^N a_n e(nx+n^2 t) 
\]
is define on the torus $\mathbb T$. In each horizontal strip 
\[
S_j=[0,1]\times \big[\frac{j-1}{N},\frac{j}{N}\big] \subset [0,1]^2, \quad j\in [1,N]\cap \mathbb Z,
\]
we pick a $\frac{1}{N}\times \frac{1}{N^2}$ box $B_j$, and define $E=\bigcup_{j=1}^N B_j$ to be the union of these boxes. Moreover, for $1\le M\le N$ and $W_M\subset [1,N]\cap \mathbb Z$, define $E_{W_M}=E_M=\bigcup_{j\in W_M}B_j$. Here $W_M$ is not an abuse of notation, it is the same subset as considered in Lemma \ref{inci lemma}. Then we have the following equivalence. 
\begin{lemma}
\begin{equation} \label{l4 replic}
\Big\|\sum_{n=1}^N a_n e(nx+n^2 t)\Big\|_{L^4(E)} \lessapprox N^{-\frac{1}{4}} \|a_n\|_{l^2}    
\end{equation}
is equivalent to 
\begin{equation}  \label{local l2}
\Big\|\sum_{n=1}^N a_n e(nx+n^2 t)\Big\|_{L^2(E_{W_M})} \lessapprox M^{\frac{1}{4}} N^{-1} \|a_n\|_{l^2}     
\end{equation}
for any $1\le M\le N$ and any $W_M$.
\end{lemma}
\begin{proof}
The local $L^2$ estimate \eqref{local l2} follows from inequality \eqref{l4 replic} by H\"older's inequality and the fact that $|E_M|=MN^{-3}$.   

To prove the $L^4$ estimate \eqref{l4 replic}, by the locally constant property, we can think of $|f|$ to be of the same size on each single box $B_j$ of dimensions $\frac{1}{N}\times \frac{1}{N^2}$. Then by dyadic pigeonholing, there exists a $\lambda\ge R^{-100}\|a_n\|_{l^2}$ and a subset $W_M\subset [1,N]$ such that  
\[
\Big\|\sum_{n=1}^N a_n e(nx+n^2 t)\Big\|_{L^4(E)}\lessapprox \Big\|\sum_{n=1}^N a_n e(nx+n^2 t)\Big\|_{L^4(E_{W_M})} 
\]
and $|f|\sim \lambda$ on  $E_{W_M}$. By the local $L^2$ estimate \eqref{local l2}, we know that 
\[
\lambda |E_{W_M}|^{\frac{1}{2}} \lessapprox M^{\frac{1}{4}} N^{-1} \|a_n\|_{l^2}=|E_{W_M}|^{\frac{1}{4}} N^{-\frac{1}{4}} \|a_n\|_{l^2}, 
\]
which is equivalent to that 
\[
\lambda |E_{W_M}|^{\frac{1}{4}} \lessapprox N^{-\frac{1}{4}} \|a_n\|_{l^2}. 
\]
This is exactly what we want in \eqref{l4 replic}. 
\end{proof}
By this lemma, Theorem \ref{L4 thm} follows from the local $L^2$-estimate \eqref{local l2}. 

\section{Proof of the local $L^2$-estimate \eqref{local l2}} 

\begin{proof}
We abbreviate $E_{W_M}$ by writing $E_M$ and recall that 
\[
E_M=\bigcup_{j\in W_M} B_j,
\]
where each $B_j$ is in the horizontal strip $S_j$. By duality, the inequality \eqref{local l2} is equivalent to 
\[
\langle \sum_{n=1}^N a_n e(nx+n^2 t), h1_{E_M} \rangle_{\mathbb T^2} \lessapprox M^{\frac{1}{4}} N^{-1} \|a_n\|_{l^2} \|h\|_{L^2}
\]
for each $h\in L^2(E_M)$ and for each sequence $\{a_n\}_{n=1}^N$. Here $1_{E_M}$ denotes the characteristic function of the set $E_M$. Again by duality, it is further equivalent to 
\[
\sum_{n=1}^N  \big|\langle e(nx+n^2 t), h1_{E_M} \rangle_{\mathbb T^2} \big|^2 \lessapprox M^{\frac{1}{2}} N^{-2} \|h\|_{L^2}^2
\]
for each $h\in L^2(E_M)$. This can be rewritten as 
\[
\langle K\ast h 1_{E_M}, h 1_{E_M} \rangle_{\mathbb T^2} \lessapprox M^{\frac{1}{2}} N^{-2} \|h\|_{L^2}^2
\]
where $K$ is the kernel defined as 
\[
K(x,t)=\sum_{n=1}^N e(nx+n^2 t)
\]
and the convolution is taken on the torus. Now we have got rid of the coefficients $\{a_n\}$ and can apply circle method to the kernel $K$. By Dirichlet's approximation, given $t\in [0,1]$, there is a unique $\frac{a}{q}$ with $(a,q)=1$, $q\le N$ such that 
\[
\big|t-\frac{a}{q} \big|\le \frac{1}{qN}. 
\]
Therefore, given a dyadic $Q\in [1,N/(\log N)^{10}]$, we define the major arcs 
\[
\mathfrak{M}_q=\sum_{\substack{0\le a<q\\(a,q)=1}} \big[ \frac{a}{q}-\frac{1}{QN},\frac{a}{q}+\frac{1}{QN} \big]
\]
for $q\sim Q$, and 
\[
\mathfrak{M}_Q=\sum_{q\sim Q}\mathfrak{M}_q. 
\]
These major arcs are disjoint from each other. And outside the major arcs,
\[
|K(x,t)|\lessapprox N^{\frac{1}{2}}. 
\]
Define $\varphi$ to be a bump function which is $=1$ on $[-1,1]$ and is compactly supported on $[-2,2]$. Then we decompose 
\[
K=\sum_{\substack{1\le Q\le N/(\log N)^{10}\\ Q \, \text{dyadic}}} K_Q+K'
\]
where 
\[
K_Q(x,t)=K(x,t)\cdot \sum_{q\sim Q} \sum_{\substack{0\le a<q\\(a,q)=1}} \varphi\big( \frac{t-a/q}{1/(QN)} \big),
\]
and 
\[
K'=K-\sum_{\substack{1\le Q\le N/(\log N)^{10}\\ Q \, \text{dyadic}}} K_Q. 
\]
Since the number of $Q$ is bounded by $2 \log N$, it suffices to show that for each $Q$, 
\begin{equation}  \label{KQ}
\langle K_Q\ast h 1_{E_M}, h 1_{E_M} \rangle_{\mathbb T^2} \lessapprox M^{\frac{1}{2}} N^{-2} \|h\|_{L^2}^2,    
\end{equation}
and 
\begin{equation} \label{K'}
\langle K'\ast h 1_{E_M}, h 1_{E_M} \rangle_{\mathbb T^2} \lessapprox M^{\frac{1}{2}} N^{-2} \|h\|_{L^2}^2.   
\end{equation}

We prove \eqref{K'} first. Since $\|K'\|_{L^\infty}\lessapprox N^{\frac{1}{2}}$, we can proceed as follows. 
\begin{equation}
\begin{split}
\langle K'\ast h 1_{E_M}, h 1_{E_M} \rangle_{\mathbb T^2}
=&\int_{E_M} \int_{E_M} K'(x-y,t-s)h(y,s)dyds \overline{h(x,t)} dxdt  
\\ \lessapprox & N^{\frac{1}{2}} \int_{E_M}\int_{E_M} |h(y,s)|dyds|h(x,t)|dxdt
\\ = & N^{\frac{1}{2}} \|h\|_{L^1(E_M)}^2 
\\ \le & N^{\frac{1}{2}} \|h\|_{L^2}^2 \cdot |E_M|   \quad (\text{H\"older's inequality})
\\ =&N^{\frac{1}{2}} MN^{-3} \|h\|_{L^2}^2
\\ \le & M^{\frac{1}{2}}N^{-2} \|h\|_{L^2}^2,
\end{split}
\end{equation}
since $M\le N$. 

Then it remains to show the estimate \eqref{KQ}. Let us make some preliminary reduction. By dyadic pigeonholing, there exist dyadic numbers $\lambda_1,\lambda_2>0$ and subsets $Y_1,Y_2\subset E_M$ such that $|h|\sim \lambda_l$ on $Y_l$, and 
\[
\begin{split}
&\big|\langle K_Q\ast h1_{E_M},h1_{E_M}\rangle \big|  \\ \lessapprox & \big| \langle K_Q\ast h1_{Y_1},h1_{Y_2} \rangle \big|.
\end{split}
\]
Also, we know that 
\begin{equation}  \label{h l2 low}
\|h\|_{L^2}^2 \gtrsim \lambda_l^2 |Y_l|,\quad l=1,2.    
\end{equation}
Expand the inner product and use H\"older's inequality twice, we have 
\begin{equation}  \label{K_Q comp}
\begin{split}
&\langle K_Q\ast h1_{Y_1}, h1_{Y_2} \rangle 
\\ =& \int_{\mathbb T^2}\int_{\mathbb T^2} K_Q(x-y,t-s)h1_{Y_1}(y,s) \overline{h1_{Y_2}}(x,t)dxdtdyds
\\ \le & \|h\|_{L^1(Y_1)}^{\frac{1}{2}} \|h\|_{L^1(Y_2)}^{\frac{1}{2}}
\\ \times & \big( \int \int |K_Q(x-y,t-s)|^2 \cdot |h1_{Y_1}(y,s)| \cdot |h1_{Y_2}(x,t)|dydsdxdt \big)^{\frac{1}{2}}
\\ \le & (\lambda_1\lambda_2)^{\frac{1}{2}}\cdot |Y_1|^{\frac{1}{4}}|Y_2|^{\frac{1}{4}} \|h\|_{L^2} 
\\ \times & \big( \int \int |K_Q(x-y,t-s)|^2 \cdot 1_{Y_1}(y,s) \cdot 1_{Y_2}(x,t)dydsdxdt \big)^{\frac{1}{2}}. 
\end{split}   
\end{equation}
We compute the inner integral. By the bound $\|K_Q\|_{L^\infty}\lesssim \frac{N}{Q^{\frac{1}{2}}}$ \cite[Lemma 3.18]{B} and the trivial estimate $|B_j|\le N^{-3}$ for each $\frac{1}{N}\times \frac{1}{N^2}$ box $B_j$, we have 
\begin{equation}  \label{K_Q comp 2}
\begin{split}
&\int \int |K_Q(x-y,t-s)|^2 \cdot 1_{Y_1}(y,s) \cdot 1_{Y_2}(x,t)dydsdxdt    
\\ =&\sum_{i,j\in W_M} \int\int |K_Q(x-y,t-s)|^2 \cdot 1_{Y_1\cap B_i}(y,s) \cdot 1_{Y_2\cap B_j}(x,t)dydsdxdt    
\\ \le & \frac{N^2}{Q} \cdot N^{-6} \cdot \#\big\{(i,j)\in W_M^2: \exists (x_i,t_i)\in B_i, (x_j,t_j)\in B_j, 
\\ \qquad & t_i-t_j=\frac{a}{q}+O(\frac{1}{QN}) \text{ for some } (a,q)=1, q\sim Q \big\}. 
\end{split}    
\end{equation}
Since each box $B_j$ has thickness $\frac{1}{N^2}$ in the $t$-direction and $\frac{1}{QN}\ge \frac{1}{N^2}$, in the above counting problem, we can think of the projection of each $B_j$ onto the $t$-axis as a discrete point, thus Lemma \ref{inci lemma} directly implies that the number of pairs $(i,j)$ is bounded by $C_\epsilon N^{\epsilon} MQ$. Insert this estimate into \eqref{K_Q comp 2} and then into \eqref{K_Q comp}, we have 
\[
\begin{split}
&\langle K_Q\ast h1_{Y_1}, h1_{Y_2} \rangle 
\\ \lessapprox & M^{\frac{1}{2}} N^{-2} (\lambda_1 \lambda_2)^{\frac{1}{2}} |Y_1|^{\frac{1}{4}}|Y_2|^{\frac{1}{4}}\|h\|_{L^2}
\\ = & M^{\frac{1}{2}} N^{-2} \big(\lambda_1^2|Y_1|\big)^{\frac{1}{4}} \big(\lambda_2^2|Y_2|\big)^{\frac{1}{4}}  \|h\|_{L^2}
\\ \lessapprox & M^{\frac{1}{2}} N^{-2}\|h\|_{L^2}^2,
\end{split}
\]
where we use \eqref{h l2 low} in the last step. Now the proof for \eqref{KQ} is established. 
\end{proof}

\section{Counterexamples}

In \cite{FRW}, Y. Fu, K. Ren and H. Wang constructed the following example:
\begin{theorem}[FRW]  \label{FRW}
For infinitely many positive integers $N$, there is a convex $C^2$ function $\Gamma$ such that 
\begin{equation}  \label{jarnick curve}
\Gamma\big(\frac{n}{N}\big) \in \frac{\mathbb Z}{N}    
\end{equation}
for $\gtrapprox N^{\frac{2}{3}}$-many integers $n\in [1,N]$. \end{theorem}

This leads to the following counterexample for Conjecture
\ref{MT general}. 

\begin{lemma}  \label{counterex}
For infinitely many integers $N=R^{\frac{1}{2}}$, there is a convex $C^2$ function $\Gamma$, an $L^2$ function $g$ and a one-dimensional weight $\omega$ such that 
\[
\int_{B_R} |E_\Gamma g|^2 \omega \gtrapprox R^{\frac{1}{12}} \sup_{T} \omega(T)^{\frac{1}{2}} \|g\|_2^2,
\]
where $T$ varies over all $R\times R^{\frac{1}{2}}$ tubes. 
\end{lemma}
\begin{proof}
Let $R^{\frac{1}{2}}=N$ be an integer considered in Theorem \ref{FRW}, and let $\Gamma$ be the corresponding convex $C^2$ function. For such $N$ and $\Gamma$, we define $\mathcal{I}$ to be the set of integers $n$ such that \eqref{jarnick curve} holds.

Define 
\[
E_\Gamma g(x,t)= \sum_{n\in \mathcal{I}} e\Big( \frac{n}{N} x+ \Gamma\big(\frac{n}{N}\big)t \Big)
\]
on $B_R$, then
\[
\|g\|_2^2\sim R^{-1} \|E_\Gamma g\|_{L^2(B_R)}^2 \approx  R^{\frac{4}{3}}. 
\]
Here $\approx$ means we have both $\lessapprox$ and $\gtrapprox$. Also, we define $\omega$ to be the characteristic function of the union of unit boxes centered at 
\[
\{(x,t)\in B_R: \frac{x}{N}\in \mathbb Z, \frac{t}{N}\in \mathbb Z\}. 
\]
By \eqref{jarnick curve}, it is easily seen that $|E_\Gamma g|=|\mathcal{I}|\approx R^{\frac{1}{3}}$ at these points, so by the locally constant property, 
\[
\int_{B_R} |E_\Gamma g|^2 \omega \approx R\cdot |\mathcal{I}|^2 \approx R^{\frac{5}{3}}. 
\]
Also, 
\[
\omega(T)\lesssim R^{\frac{1}{2}}
\]
for any $R\times R^{\frac{1}{2}}$ tube $T$. Then by simple computation, 
\[
\int_{B_R} |E_\Gamma g|^2 \omega \gtrapprox R^{\frac{1}{12}} \sup_{T} \omega(T)^{\frac{1}{2}} \|g\|_2^2.
\]
\end{proof}


\begin{thebibliography}{10}

\bibitem{BM} T. F. Bloom and J. Maynard, {\it A new upper bound for sets with no square differences}. Compos. Math. 2022; 158(8):1777-1798.

\bibitem{B} J. Bourgain, {\it Fourier transform restriction phenomena for certain lattice subsets and applications to nonlinear evolution equations Part I: Schr\"odinger equations}. GAFA. 3, 107–156 (1993).

\bibitem{BDdecoupling} J. Bourgain and C. Demeter, {\it The proof of the $l^2$ Decoupling Conjecture}. Ann. of Math. (2), \textbf{182} (2015), 351-389.

\bibitem{CIW} A. Carbery, M. Iliopoulou and H. Wang. {\it Some sharp inequalities of Mizohata-Takeuchi-type}. Rev. Mat. Iberoam. 40 (2024), no. 4, pp. 1387–1418.

\bibitem{FRW} Y. Fu, K. Ren and W. Wang, {\it A note on maximal operators for the Schr\"{o} dinger equation on $\mathbb {T}^ 1$}. arXiv preprint arXiv:2307.12870, 2023.

\bibitem{HC} H. Cairo, {\it A counterexample to the Mizohata-Takeuchi conjecture }.  arXiv: 2502.06137, 2025.

\bibitem{G} L. Guth, {\it An enemy scenario in restriction theory}. Joint talk for AIM Research Community {\it Fourier restriction conjecture and related problems} and HAPPY network (2022),
https://www.youtube.com/watch?v=x-DET83UjFg.

\bibitem{GIOW} L. Guth, A. Iosevich, Y. Ou and H. Wang, {\it On Falconer's distance set problem in the plane}. Invent. Math. \textbf{219,} 779–830 (2020). 

\bibitem{O} A. Ortiz, {\it A sharp weighted Fourier extension estimate for the cone in $\mathbb R^3$ based on circle tangencies}. To appear in J. Anal. Math.. 







\end{thebibliography}
\end{document}